\documentclass[11pt,reqno]{amsart} 

\usepackage{amssymb,latexsym}
\usepackage[draft=true]{hyperref}
\usepackage{cite} 

\usepackage[height=190mm,width=130mm]{geometry} 

\theoremstyle{plain}
\newtheorem{theorem}{Theorem}
\newtheorem{lemma}{Lemma}
\newtheorem{corollary}{Corollary}

\newtheorem{claim}[theorem]{Claim}

\theoremstyle{definition}
\newtheorem{definition}{Definition}

\theoremstyle{remark}
\newtheorem{remark}{Remark}

\numberwithin{equation}{section} 

\begin{document}
\title[A meeting between the $P^{3}_{1}$-set and the cubic-triangular numbers]{An unexpected meeting between the $P^{3}_{1}$-set and the cubic-triangular numbers} 

\author{A. DEBBACHE}
\address{USTHB University, Faculty of Mathematics, LATN Laboratory\\ P.O. Box 32, El Alia 16111, Bab-Ezzouar, Algiers, Algeria,}

\email{a\_debbache2003@yahoo.fr}

\thanks{$^*$ Corresponding author.}

\author{S. BOUROUBI$^*$}

\address{USTHB University, Faculty of Mathematics, L'IFORCE Laboratory\\ P.O. Box 32, El Alia 16111, Bab-Ezzouar, Algiers, Algeria,}
\curraddr{Other Institution\\ Some Other Department\\ P.~O. Box 0002 \\ 12347
  City\\ Other Country}

\email{sbouroubi@usthb.dz\ or\ bouroubis@gmail.com}

\begin{abstract}
A set of $m$ positive integers $\{x_{1},\ldots,x_{m}\}$ is called a $P^{3}_{1}$-set of size $m$ if the product of any three elements in the set increased by one is a cube integer. A $P^{3}_{1}$-set $S$ is said to be extendible if there exists an integer $y\not\in S$ such that $S\cup\{y\}$ still a $P^{3}_{1}$-set. Now, let consider the Diophantine equation $u(u+1)/2=v^{3}$ whose integer solutions produce what we called cubic-triangular numbers. The purpose of this paper is to prove simultaneously that the $P^{3}_{1}$-set $\{1,2,13\}$ is non-extendible and $n=1$ is the unique cubic-triangular number by showing that the two problems meet on the Diophantine equation $2x^{3}-y^{3}=1$ that we solve using $p$-adic analysis.
\end{abstract}


\subjclass[2010]{11D45, 11D09}

\keywords{$P^{3}_{1}$-set, cubic-triangular number, Diophantine equation}

\maketitle

\section{Introduction}
A set of $m$ positif integers $\{x_{1},\ldots,x_{m}\}$ is called a Diophatine $m$-tuple or a $D(1)$-$m$-tuple, if the product of any two elements in the set increased by one is a perfect square, i.e., \mbox{$x_{i}x_{j}+1=u^{2}_{ij}$}, where $u_{ij}\in\mathbb{N}^{*}$, for $1\leq i<j\leq m$. Diophantus of Alexandria was the first to look for such sets. He found a set of four positive rational numbers with the above property $\left\{\frac{1}{16},\frac{33}{16},\frac{17}{4},\frac{105}{16}\right\}$. \mbox{However}, Fermat was the first to give $\left\{1,3,8,120\right\} $ as an example of a Diophantine quadruple. For a detailed history on Diophantine $m$-tuples and its results, we refer the reader to Dujella's webpage \cite{Du}. Throughout the following we consider in a similar way what we have called a $P^{3}_{1}$-set.
\section{Definitions}
\begin{definition}
A $P^{3}_{1}$-set of size $m$ is a set $S=\{x_{1},\ldots,x_{m}\}$ of distinct positive integers, such that \mbox{$x_{i}x_{j}x_{k}+1$} is a cube for $1\leq i<j<k\leq m$.
\end{definition}
\begin{definition}
A $P^{3}_{1}$-set $S$ is said to be extendible if there exists an integer $y\not\in S$ such that $S\cup\{y\}$ is a $P^{3}_{1}$-set.
\end{definition}
\begin{definition}
A triangular number is a figurate number that can be represented in the form of an equilateral triangle of points, where the first row contains a single element and each subsequent row contains one more element than the previous one. Let $T_{n}$ denotes the $n^{th}$ triangular number, then $T_{n}$ is equal to the sum of the $n$ natural numbers from 1 to $n$, whose initial values are listed as the sequence A000217 in \cite{Ol}.
$$T_{n}=\frac{n(n+1)}{2}=\dbinom{n+1}{2},$$
where $\dbinom{n}{k}$ is a binomial coefficient.
\end{definition}
\begin{definition}
A cubic-triangular number is a positive integer that is simultaneously cubic and triangular. Such a number must satisfy $T_{n}=m^{3}$ for some positive integers $n$ and $m$, so
\begin{equation}\label{eq1}
 \frac{n(n+1)}{2}=m^{3}.
\end{equation}
\end{definition}
\section{Some Claims}
\begin{claim}\label{clm1}
The triple $\left\{a-1,a+1,a^{4}+a^{2}+1\right\}$ is an infinite family of $P^{3}_{1}$-set for any positive integers $a\geq2$.
\end{claim}
\begin{proof}
Thanks to the identity $x^{3}=(x-1)(x^{2}+x+1)+1$, it is enough to substitute $x$ by $a^{2}$ to get, $$a^{6}-1=(a-1)(a+1)(a^{4}+a^{2}+1).$$
\end{proof}
\begin{claim}\label{clm2}
The triple $\{a,b,a^{2}b^{2}+3ab+3\}$ form an infinite family of $P^{3}_{1}$-set for any positive integers \mbox{$a$ and $b$}, such that $1\leq a<b$.
\end{claim}
\begin{proof}
The result follows thanks to the identity : $$(ab+1)^{3}-1=ab(a^{2}b^{2}+3ab+3).$$
\end{proof}
\begin{remark}
The triple $\{ 1,2,13\}$ is a $P^{3}_{1}$-set, it belongs to the family in \mbox{Claim \ref{clm2}}, for $a=1$ and $b=2$.
\end{remark}
\section{Main Results}
\begin{theorem}\label{th1}
Any  $P^{3}_{1}$-set is finite.
\end{theorem}
\begin{proof}
Let $S=\{x_{1},x_{2},x_{3},\ldots,x_{m}\}$ be a $P^{3}_{1}$-set. Suppose that $S\cup\{y\}$ still a $P^{3}_{1}$-set, then by setting
\begin{equation*}
\begin{cases}
a=x_{m}x_{m-1}, \cr
b=x_{m}x_{m-2},\cr
c=x_{m-1}x_{m-2},
\end{cases}
\end{equation*}
we get an elliptic curve
$$(ay+1)(by+1)(cy+1)=t^{3},$$
which has only finitely many integral solutions \cite{Sh}.
\end{proof}
In the following, we will restrict our attention to equation (\ref{eq2}), for which we present a proof for it's uniqueness integer solution, by using $p$-adic analysis tools.
\begin{equation}\label{eq2}
 2x^{3}-y^{3}=1.
\end{equation}
We first briefly remind Hensel's Lemma and Strassman's Theorem \cite{Ca}.
\begin{lemma}\label{lem1} \textsc{(Hensel)} Let $P\left( X\right) \in\mathbb{Z}_{p}\left[ X\right] $, a monic polynomial. Suppose that $x\in\mathbb{Z}_{p}$, satisfied:
\begin{enumerate}
  \item $P\left( x\right) \equiv 0\ (mod\ p)$,
  \item $\dfrac{d}{dx}\left(P(x)\right) \not\equiv 0\ (mod\ p)$.
\end{enumerate}
So, there is a unique  $y\in\mathbb{Z}_{p}$ such as $P\left( y\right) =0$ and $y\equiv x\ (mod\ p)$.
\end{lemma}
\begin{theorem} \textsc{(Strassman)} Let  $\mathbb{K}$ be a complete field for the non-archimedean norm $\| . \|$, $\mathcal{A}$ its ring of integers, and let
$$g\left(x\right) =\overset{+\infty }{\underset{n=0}{\sum }}g_{n}x^{n}.$$
Suppose that $g_{n}\rightarrow 0$ (so $g\left(x\right)$ converges in $\mathcal{A}$), but with $g_{n}$ not all zeros, there is at most a finite number of elements $b$ of $\mathcal{A}$ such that $g\left( b\right) =0$. More precisely, there is at most $M$ elements $b$ of $\mathcal{A}$, such that
$$\| g_{N} \| =\underset{n}{max}\| g_{n}\|,\ \ \| g_{n} \|<\| g_{N}\|, \forall n>M.$$
\end{theorem}
\begin{lemma}\label{lem2}
Let $b\in\mathbb{Q}_{p}, \left\vert b\right\vert_{2}\leq 2^{-2}$ and $\left\vert b\right\vert_{p}\leq p^{-1}$ ($p\neq 2$). So, there is a series \mbox{$\Phi_{b}\left(X\right) =\underset{n\geq 0}{\sum }\gamma _{n}X^{n}$}, with $\gamma _{n}\in\mathbb{Q}_{p}$, $\gamma_{n}\rightarrow 0$, such that $\Phi _{b}\left( r\right) =\left(1+b\right)^{r}$, $\forall r\in\mathbb{Z}$.
\end{lemma}
\noindent By application of these results, we may show
\begin{theorem}\label{th2}
The unique positive integer solution of Equation (\ref{eq2}) is $(1,1)$.
\end{theorem}
\begin{proof}
Let $(a,b)$ be a solution of Equation (\ref{eq2}). As $N(2^{\frac{1}{3}}a-b)=2a^{3}-b^{3}$, so $2^{\frac{1}{3}}a-b$ is an algebraic unit, especially $\theta=\theta _{1}=2^{\frac{1}{3}}-1$. For convenience, on all the rest of the paper we will work  with the field $\mathbb{K}=\mathbb{Q}(\theta)$. Note $\mathcal{A}$, the integer ring of $\mathbb{K}$ and $\mathcal{U}$ its group of units. We have
$$2=(\theta +1)^{3}.$$
Hence
$$\theta^{3}+3\theta^{2}+3\theta-1=0.$$
Let $f(X)=X^{3}+3X^{2}+3X-1$ be the irreducible polynomial of $\theta$. According to Dirichlet's unit theorem $\mathcal{U}=G\times\mathbb{Z}^{r+s-1}$, where $G$ is the root group of the unit of $\mathcal{A}$, $r$ is the number of real zeros of $f$, and $2s$ is the number of complex zeros of $f$. Here, $r=s=1$, so $\mathcal{U}=G\times\mathbb{Z}$. Let $\alpha$ be a root of the unit of $\mathcal{A}$, then the dimension of the field $\mathbb{Q}(\alpha)$ divides the dimension of the field $\mathbb{K}$, so it is a divisor of $3$. Since $\mathbb{Q}(\alpha) \neq \mathbb{K}$, we have $\mathbb{Q}(\alpha)=\mathbb{Q}$ and $A=\mathbb{Z}$, where $A$ is the ring of the units of $\mathbb{Q}(\alpha)$. However, the only invertible elements of $\mathbb{Z}$ are $+1$ and $-1$, hence $\mathcal{U}=\{\pm u^{n}/n\in \mathbb{Z}\}$. Let $u>1$ be the fundamental unit, $\rho e^{i\theta}$ and $\rho e^{-i\theta}$ its conjugates. We have
$$N(u)=u\times \rho e^{i\theta}\times \rho e^{-i\theta}=1.$$
It follows
$$u=\rho^{-2}.$$
In addition
$$disc_{\mathbb{Z}}(u)=(u-\rho e^{i\theta })^{2}(u-\rho e^{-i\theta })^{2}(\rho e^{i\theta}-\rho e^{-i\theta})^{2}=-4(\rho^{3}+\rho^{-3}-2cos\theta)^{2}sin^{2}\theta.$$
For $c=cos\theta$, let us set $g(x)=(1-c^{2})(x-2c)^{2}-x^{2}$.  Then we get
$$g(x)\leq 4(1-c^{2}),$$
or even
$$(1-c^{2})(x-2c)^{2}\leq x^{2}+4(1-c^{2}).$$
Replacing $x$ by $\rho^{3}+\rho^{-3}$, we obtain
$$(1-c^{2})(\rho^{3}+\rho^{-3}-2c)^{2}<u^{3}+u^{-3}+6.$$
This involves that
\begin{center}
$\left\vert disc_{\mathbb{Z}}(u)\right\vert <4(u^{3}+u^{-3}+6)$.
\end{center}
Therefore
\begin{center}
$u^{3}>\dfrac{d}{4}-6-u^{-3}>\dfrac{d}{4}-7,$
\end{center}
where $d=\left\vert disc_{\mathbb{Z}}(u)\right\vert $.\\
The discriminant of $f$ equals -108, then $u^{3}>20$. Hence
$$u>2,7144.$$%
Since $\theta^{-1}\simeq 3,8473$, and $u^{2}>7,3680$, we get $u=\theta^{-1}$. We therefore have
$$\mathcal{U}=\{\pm \theta^{n}/n\in\mathbb{Z}\}.$$
Moreover, $2^{\frac{1}{3}}a-b=(a-b)+(2^{\frac{1}{3}}-1)a=(a-b)+a\theta$. Since $N(2^{\frac{1}{3}}a-b)=1$, we have $(a-b)+a\theta\in\mathcal{U}$, i.e., there exists $n\in\mathbb{Z}$ such that
\begin{equation}\label{eq3}
(a-b)+a\theta=\pm \theta^{n}.
\end{equation}
If we take for instance $(a-b)+a\theta=\theta^{n}$, then we get
$$(a-b)+a\theta_{i}=\theta_{i}^{n},\ \textrm{pour}\ i=1,2,3,$$
where $\theta_{1}=\theta$, $\theta_{2}$ and $\theta_{3}=\overline{\theta_{2}}$ are the three zeros of $f$. We have obviously
\begin{equation}\label{eq4}
\frac{1}{f^{^{\prime}}(\theta_{1})}+\frac{1}{f^{^{\prime}}(\theta_{2})}+\frac{1}{f^{^{\prime}}(\theta_{3})}=\frac{1}{(\theta_{1}-\theta_{2})
(\theta_{1}-\theta_{3})}+\frac{1}{(\theta_{2}-\theta_{1})(\theta_{2}-\theta_{3})}+\frac{1}{(\theta_{3}-\theta_{1})(\theta_{3}-\theta_{2})}=0,
\end{equation}
and
\begin{equation}\label{eq5}
\frac{\theta_{1}}{f^{^{\prime }}(\theta_{1})}+\frac{\theta_{2}}{f^{^{\prime }}(\theta_{2})}+\frac{\theta_{3}}{f^{^{\prime }}(\theta_{3})}=\frac{\theta_{1}}{(\theta_{1}-\theta_{2})(\theta_{1}-\theta_{3})}+\frac{\theta_{2}}{(\theta_{2}-\theta_{1})
(\theta_{2}-\theta_{3})}+\frac{\theta_{3}}{(\theta_{3}-\theta_{1})(\theta_{3}-\theta_{2})}=0.
\end{equation}
If we multiply (\ref{eq4}) by $a-b$ and (\ref{eq5}) by $a$, we find
$$\dfrac{(a-b)+a\theta_{1}}{f^{^{\prime}}(\theta_{1})}+\dfrac{(a-b)+a\theta_{2}}{f^{^{\prime}}(\theta_{2})}+\dfrac{(a-b)+a\theta_{3}}{f^{^{\prime}}
(\theta_{3})}=\dfrac{\theta_{1}^{n}}{f^{^{\prime}}(\theta_{1})}+\dfrac{\theta_{2}^{n}}{f^{^{\prime}}(\theta_{2})}+\dfrac{\theta_{3}^{n}}{f^{^{\prime}}
(\theta_{3})}=0.$$
So, solving the equation $2x^{3}-y^{3}=1$, is like finding the zeros of the sequence $(c_{n})_{n\in\mathbb{Z}}$ defined by:
$$c_{n}=\dfrac{\theta_{1}^{n}}{(\theta_{1}-\theta_{2})(\theta_{1}-\theta_{3})}+\dfrac{\theta_{2}^{n}}{(\theta_{2}-\theta_{1})
(\theta_{2}-\theta_{3})}+\dfrac{\theta_{3}^{n}}{(\theta_{3}-\theta_{1})(\theta_{3}-\theta_{2})}\cdot\vspace{0.2cm}$$
Now, let us work locally in $\mathbb{Q}_{p}$. For this purpose, we are looking for an adequate prime number $p$ that allows us to apply Hensel's Lemma in ordre to find two zeros of $f$ $\alpha$ and $\beta \in\mathbb{Q}_{p}$, the third one is then given by $\alpha +\beta +\gamma=-3$. Since $f(3)=2\times 31\equiv 0\ (mod\ 31)$, $f(6)=11\times 31\equiv 0\ (mod\ 31)$, $f'(3)=48\not\equiv 0\ (mod\ 31)$ and \mbox{$f'(6)=147\not\equiv 0\ (mod\ 31)$}, then according to Hensel's Lemma, there exist a unique $\alpha$ and $\beta$ in $\mathbb{Z}_{31}$, where $\alpha = 34\ \textrm{and}\ \beta = 37$, hence $\gamma =-74$.
According to Fermat's little theorem, we have $\alpha^{30}\equiv 1\ (mod\ 31)$. Thus $\alpha^{30}=1+a$. Since $$\alpha^{30}=34^{30}\equiv 838\ (mod\ 31^{2}).$$
Then $$a\equiv 837\ (mod\ 31^{2}).$$
Similarly,
$$\beta^{30}\equiv 1\ (mod\ 31).$$
Thus
$$\beta^{30}=1+b.$$
Since
$$\beta^{30}=37^{30}\equiv 869\ (mod\ 31^{2}).$$ Then
$$b\equiv868\ (mod\ 31^{2}).$$
Likewise, $\gamma^{30}\equiv 1\ (mod\ 31)$. Thus
$$\gamma^{30}=1+c.$$
Since $$\gamma^{30}=74^{30}\equiv 94\ (mod\ 31^{2}).$$
Then $$c\equiv93\ (mod\ 31^{2}).$$
In the rest of the proof we will need the following table:
\begin{center}
\begin{tabular}{|c|c|c|c|}
\hline
$r$ & $\alpha^{r}\ (mod\ 31^{2})$ & $\beta^{r}\ (mod\ 31^{2})$ & $\gamma^{r}\ (mod\ 31^{2})$ \\ \hline
$1$ & $34$ & $37$ & $-74$ \\ \hline
$30$ & $838$ & $869$ & $94$ \\ \hline
\end{tabular}
\end{center}
In addition, we have
$$\begin{array}{lll}
c_{r+30s} & = & \dfrac{\alpha^{r}}{(\alpha-\beta)(\alpha-\beta)}(\alpha^{30})^{s}+\dfrac{\beta^{r}}{(\beta-\alpha)(\beta-\gamma)}(\beta^{30})^{s}+\dfrac{\gamma^{r}}{(\gamma-\beta)
 (\gamma-\alpha)}(\gamma^{30})^{s} \vspace{0.2cm}\\
          & = & \dfrac{\alpha^{r}}{(\alpha-\beta)(\alpha-\beta)}(1+a)^{s}+\dfrac{\beta^{r}}{(\beta-\alpha)(\beta-\gamma)}(1+b)^{s}+\dfrac{\gamma^{r}}{(\gamma-\beta)(\gamma-\alpha)}
          (1+c)^{s}\vspace{0.2cm}\\
          &\equiv& c_{r}\ (mod\ 31),\ \textrm{for}\ 1\leq r\leq 30.
\end{array}$$
The calculations show that $c_{r}\neq 0$ for $r\neq 1,30$. Since, $c_{r+30s}\equiv c_{r}\ (mod\ 31)$, we get
$$c_{r+30s}\neq 0, \forall s\in\mathbb{N},\ \textrm{for}\ r\neq 1, 30.$$
Let's say for $r= 1, 30$ and $s\in \mathbb{Q}_{31}$,
$$u_{r}(s)=\dfrac{\alpha^{r}}{(\alpha-\beta)(\alpha-\beta)}(1+a)^{s}+\dfrac{\beta^{r}}{(\beta-\alpha)(\beta-\gamma)}(1+b)^{s}+\dfrac{\gamma^{r}}{(\gamma-\beta)
 (\gamma-\alpha)}(1+c)^{s}.$$
To demonstrate the result, it is enough to work only with $u_{1}$ and $u_{30}$. Since $\left\vert a\right\vert_{31}\leq 31^{-1}, \left\vert b\right\vert_{31}\leq 31^{-1}$ and $\left\vert c\right\vert_{31}\leq 31^{-1}$, we deduce from Lemma \ref{lem2} that $u_{r}$ is a function that we can develop as a series:
$$\lambda _{0,r}+\lambda _{1,r}s+\lambda _{2,r}s^{2}+\cdots$$
We have
$$\ \lambda _{0,r}=0,\ \textrm{for}\ r=1,30,\ \ \lambda _{j,r}\not\equiv 0\ \left(mod\ 31^{2}\right),\ \textrm{for}\ j\geq 2,\ \ r\ \textrm{unspecified},$$
and
$$\lambda_{1,r}=\dfrac{\alpha^{r}}{(\alpha-\beta)(\alpha -\gamma)}a+\dfrac{\beta^{r}}{(\beta-\alpha)(\beta-\gamma)}b+\dfrac{\gamma^{r}}{(\gamma-\beta)(\gamma-\alpha)}c\not\equiv 0\ (mod\ 31^{2}),$$
for $r=1,30.$\\

According to Strassman's theorem, the functions $u_{r}(s),\  r=1,30$, have at most one root. As they have at least one root, they have therefore exactly one root. From Equation (\ref{eq3}), we have $\theta^{0}=1$ implies $a=0$ and $b=-1$, which is impossible, and $\theta^{1}=\theta$ implies $a=b=1$.\\

This completes the proof.
\end{proof}
\begin{corollary}\label{cor1}
The $P^{3}_{1}$-set $\{1,2,13\}$ is nonextendible.
\end{corollary}
\begin{proof}
Suppose there exists an integer $d>13$ such that the quadruple $\{1,2,13,d\}$ is a $P^{3}_{1}$-set. Then the following system of equations has an integral solution $(u,v,w)\in\mathbb{N}^{3}$:
\begin{equation*}
(S)
\begin{cases}
2d+1=u^{3}, \cr
13d+1=v^{3},\cr
26d+1=w^{3}.
\end{cases}
\end{equation*}
\noindent The system $(S)$ yields
\begin{equation}\label{eq6}
 2v^{3}-w^{3}=1.
\end{equation}
From Theorem \ref{th2}, the unique positive integer solution of Equation (\ref{eq6}) is $(v,w) = (1, 1)$, which is impossible in $(S)$.\vspace{0.2cm}\\
This completes the proof.
\end{proof}
\begin{corollary}\label{cor2}
The unicity of positive integer solution of Equation (\ref{eq2}) implies the unicity of a cubic-triangular number.
\end{corollary}
\begin{proof}
Let $n$ be a cubic-triangular number. Since $n$ and $n+1$ are coprime then according to Equation (\ref{eq1}), there exists $x$ and $y$ two positive integers such that $m=xy$, $n=y^{3}$ and $n+1=2x^{3}$, which implies Equation (\ref{eq2}), that has from Theorem \ref{th2}, $(x,y)=(1,1)$ as unique positive integer solution. Thus, $n=1$ is the unique cubic-triangular number.
\end{proof}
\begin{remark}
As we can see, the resolution of Equation (\ref{eq2}) meets the two problems mentioned above that seem to be a priori different.
\end{remark}
\section{Conclusion}
The interest of this work is twofold. Firstly, we showed an unexpected link between two problems, which were a priori distinct. Secondly, we presented a proof for the uniqueness of the positive integer solution of the Diophantine equation $2x^{3}-y^{3}=1$, using $p$-adic analysis tools.

\bibliography{Bouroubi}
\bibliographystyle{mmn}

\end{document}